\documentclass[11pt]{article}
\usepackage[english]{babel}
\usepackage{newlfont}
\usepackage{amsbsy}
\usepackage[dvips]{graphicx}
\usepackage{color}
\usepackage{bbm}
\usepackage{graphicx, subfigure}
\usepackage{color}
\usepackage[center]{caption2}
\usepackage{amssymb}
\usepackage{amsmath}
\usepackage{latexsym}
\usepackage{amsthm}
\usepackage{amsthm}
\theoremstyle{plain}
\newtheorem{thm}{Theorem}

\newtheorem{nota}[thm]{Notation}
\newtheorem{rem}[thm]{Remark}

\newcommand{\N}{\mathbb{N}}

\usepackage{mathtools}
\def\multiset#1#2{\ensuremath{\left(\kern-.2em\left(\genfrac{}{}{0pt}{}{#1}{#2}\right)\kern-.2em\right)}}
\usepackage[center]{caption2}

\begin{document}

\title{On completions of symmetric and antisymmetric 
block diagonal partial matrices}
\author{Elena Rubei}
\date{}
\maketitle

\def\thefootnote{}
\footnotetext{ \hspace*{-0.36cm}
{\bf 2010 Mathematical Subject Classification:} 15A83

{\bf Key words:} partial matrices, completion, diagonal blocks,  
symmetric and antisymmetric matrices}

\begin{abstract}
A partial matrix is a 
matrix where  only  some of the entries are given. 
We determine the maximum rank of the symmetric completions of a
symmetric partial matrix where only the diagonal blocks are given and
the minimum rank and the maximum rank of the antisymmetric completions of an
antisymmetric partial matrix where only the diagonal blocks are given. 
\end{abstract}

\section{Introduction} Let $K$ be a field. A partial matrix over $K$ is a 
matrix where only some of the entries are given and they are elements of $K$.
A completion of a partial matrix is a specification of the unspecified 
entries. We say that a submatrix of a partial matrix is specified if all its 
entries are given.
The problem of determining whether, given a partial 
matrix, a completion  with some prescribed property exists and 
related problems have been widely studied: we quote, for instance, 
the papers \cite{BW},
\cite{BHZ}, \cite{CJRW}, \cite{Gee}, \cite{Fie-Mar}, \cite{MQ}, 
\cite{Woe},

In \cite{CJRW}, Cohen,  Johnson,  Rodman and Woerdeman determined 
the maximum rank of the completions of a partial
matrix in terms of the ranks and the sizes of its maximal  
specified submatrices.
From the results in \cite{CJRW}, we get easily the minimum rank and the 
maximum rank of the completions of a partial matrix where only the diagonal 
blocks are given:

\begin{thm} \label{matqualsiasi} (See \cite{CJRW}.)
Let $n_1,...., n_k$ be  nonzero 
natural numbers with $ n_1 \leq n_2 \leq .... \leq n_k$.  
Let $A_i \in M(n_i \times n_i, K) $ for $i=1,..., k$ and $r_i = rank (A_i)$.
Let $A $ be the partial matrix where only the 
diagonal blocks are given and whose diagonal blocks are $A_1,..., A_k$.
Then we have:

(i) the minimum of $\{rk(\tilde{A}) | \; \tilde{A} \; \mbox{\it completion of } A \}$ is $$ max\{r_i| \; i=1,...,k\}$$ 

(ii) the maximum of $\{rk(\tilde{A}) | \; \tilde{A} \; \mbox{\it completion of } A \}$ is $$ min \left\{ \sum_{i=1,...,k} n_i , \; 2 \left(
 \sum_{i=1,...,k-1} n_i \right) +r_k
\right\}.$$ 
\end{thm}

Here we determine the maximum rank of the symmetric completions of a
symmetric partial matrix where only the diagonal blocks are given 
(see Theorem \ref{matsim}) and
the minimum rank and the maximum rank of the antisymmetric completions of an
antisymmetric partial matrix where only the diagonal blocks are given
(see Theorem \ref{matantisim}). 
In \cite{Cain}, \cite{Ca-deSa}, \cite{Tian}, the analogous problem 
 has been solved for hermitian matrices.

\section{Notation}

\begin{nota}
$\bullet$ We say that a diagonal matrix $A \in M(n \times n, K)$ 
is {\em b-diagonal} if all the nonzero elements 
of the diagonal are at the beginning, that is either it is the zero  matrix or
 there 
exists $r \in \{1,...,n\}$   such that $a_{i,i} \neq 0$ if and only if 
 $i \in \{1,..., r\}$.

$\bullet $ We say that {\em  a sequence of elementary operations} on the rows and on the columns of a square matrix is {\em symmetric} if it is given by an elementary operation on the rows, the same elementary operation on the columns, another elementary operation on the rows, the same elementary operation on the columns and so on.

$\bullet $ 
For any $r, m,n \in \N-\{0\}$ with $ r \leq min\{m, n\}$, 
we define $T^r_{m,n}$ to be the matrix $ m \times n$
with entries in $K$ such that 
$$(T^r_{m,n})_{i,j} = \left\{ \begin{array}{ll}
1 & \mbox{\it if }  (i,j) = (m,n),  (m-1,n-1),...., (m-r+1, n-r+1), \\
0 & \mbox{\it otherwise,} 
\end{array}
\right.$$
for any $ i \in \{1,..., m\}$ and $j \in \{1,..., n\}$.

$\bullet$ 
For any $r, m, n \in \N-\{0\}$ with $ r \leq min\{m, n\}$, 
we define $E^r_{m,n}$ to be the matrix $m \times n$ with entries in 
$K$ such that 
$$(E^r_{m,n})_{i,j} = \left\{ \begin{array}{ll}
1 & \mbox{\it if }  (i,j) = (1,1),...., (r, r), \\
0 & \mbox{\it otherwise,} 
\end{array}
\right.$$
for any $ i \in \{1,..., m\}$ and $j \in \{1,..., n\}$.

$\bullet$ 
For any $r, m, n \in \N-\{0\}$ with $ r \leq min\{m, n\}$ and $r$ even, 
we define $R^r_{m,n}$ to be the matrix $m \times n$ with entries in 
$K$ such that 
$$(R^r_{m,n})_{i,j} = \left\{ \begin{array}{ll}
1 & \mbox{\it if }  (i,j) = (1,2), (3,4),...., (r-1, r), \\
-1 & \mbox{\it if }  (i,j) = (2,1), (4,3),...., (r, r-1), \\
0 & \mbox{\it otherwise,} 
\end{array}
\right.$$
for any $ i \in \{1,..., m\}$ and $j \in \{1,..., n\}$.

We define $T^0_{m,n}$, $E^0_{m,n}$  and $R^0_{m,n}$ 
to be the zero matrix $ m \times n$.

We write $E^r_{n}$ instead of $E^r_{n,n}$ and $R^r_{n}$ instead of
$R^r_{n,n}$ for simplicity. 
We omit the  subscript in  $T^r_{m,n}$, $E^r_{m,n}$  and $R^r_{m,n}$ 
when their size  is clear from the context.

\end{nota}

{\bf Examples.} $$T^{3}_{4,5} = 
\left(\begin{array}{ccccc}
0 & 0 & 0 & 0 & 0 \\
0 & 0 & 1 & 0 & 0 \\
0 & 0 & 0 & 1 & 0 \\
0 & 0 & 0 & 0 & 1 
 \end{array} \right), \;\;\;  E^{2}_{4} = E^{2}_{4,4} = 
\left(\begin{array}{cccc}
1 & 0 & 0 & 0  \\
0 & 1 & 0 & 0  \\
0 & 0 & 0 & 0  \\
0 & 0 & 0 & 0  
 \end{array} \right), \;\;\;  R^{4}_{5,6} = 
\left(\begin{array}{cccccc}
0  & 1 &  0 & 0 & 0 & 0\\
-1 & 0 &  0 & 0 & 0 & 0\\
0  & 0 &  0 & 1 & 0 & 0\\
0  & 0 & -1 & 0 & 0 & 0 \\
0  & 0 &  0 & 0 & 0 & 0 
 \end{array} \right).$$

\begin{nota}
Let $n_1,...., n_k$ be nonzero natural numbers.  
Let $A_i \in M(n_i \times n_i, K) $ for $i=1,..., k$.
 
We denote  by
   $Diag (A_1,...., A_k)$ the block diagonal matrix whose diagonal blocks
are $A_1,..., A_k$ (thus the entries out of the diagonal blocks are $0$). 

We denote by   $diag (A_1,...., A_k)$  the partial matrix where only the 
diagonal blocks are given and whose diagonal blocks are $A_1,..., A_k$.
We call such a matrix a block diagonal  partial matrix.
\end{nota}

\section{Completions of  symmetric  block diagonal partial matrices}

\begin{rem} \label{oss} Let
 $K$ be a field of characteristic different from $2$.
Let $n_1,...., n_k$ be nonzero natural numbers.  
Let $A_i \in M(n_i \times n_i, K) $ be symmetric
 for $i=1,..., k$ and $r_i = rank (A_i)$.
It is well known that, 
by a symmetric sequence of elementary operations $h_i$, 
 we can change the matrix $A_i$ into a b-diagonal matrix $D_i$.

Then there exists a completion  of rank $l$ of the partial matrix 
$diag (A_1,...., A_k)$
if and only if there exists a completion  of rank $l$ of the partial matrix 
$diag (D_1,...., D_k)$.
\end{rem}

\begin{proof} Let 
\begin{center}
$I_1=\{1,..., n_1\}, \; \; I_2 =\{ n_1 + 1,...., n_1+ n_2\},\; \;....,\;\; 
I_k =\left\{\left(\sum_{i=1,..., k-1} n_i \right)+ 1,...,\sum_{i=1,..., k}
 n_i \right\}.$
\end{center}
Suppose there exists a symmetric completion of the partial matrix 
$diag (A_1,...., A_k)$ of rank $l$.  
We apply the symmetric sequence of  elementary operations $h_1$ to the rows 
and the columns with index in $I_1$,
then we apply the symmetric sequence of  elementary operations $h_2$
 to the rows  and  the columns with index in $I_2$ 
and   so on. In this way we get a symmetric completion of 
$diag (D_1,...., D_k)$ of rank $l$.

Conversely, suppose there exists a symmetric completion of 
$diag (D_1,...., D_k)$ of rank $l$. 
Apply the symmetric sequence of elementary operations $h_1^{-1}$ to the 
rows and to the columns with index in $I_1$,
then we apply the symmetric sequence of
 elementary operations $h_2^{-1}$ to the rows 
 and  the columns with index in $I_2$ 
and   so on. In this way we get a symmetric completion of 
$diag (A_1,...., A_k)$ of rank $l$. 
\end{proof}

\begin{thm} \label{matsim} Let
 $K$ be a field of characteristic different from $2$ and
let $n_1,...., n_k$ be nonzero natural numbers with
$n_1 \leq n_2 \leq .... \leq n_k$.  
Let $A_i \in M(n_i \times n_i, K) $ for $i=1,..., k$ be symmetric matrices
 and let  $r_i = rank (A_i)$.
Let $A $ be the partial matrix $diag (A_1,...., A_k)$. 
Then the maximum of $\{rk(\tilde{A}) | \; \tilde{A} \; \mbox{\it symmetric 
completion of } A \}$ is 
\begin{equation} \label{formula} 
min \left\{ \sum_{i=1,...,k} n_i , \; 2 \left( 
\sum_{i=1,...,k-1} n_i \right) +r_k
\right\}.\end{equation}
\end{thm}

\begin{proof} 
By Theorem \ref{matqualsiasi}, any completion of $A$ has rank 
less or equal than the number in (\ref{formula})
Note that this can be proved  directly in an easy way:
clearly any completion of $A$ has rank  less or equal than 
$\sum_{i=1,...,k} n_i $; besides  any completion of $A$ has rank
 less or equal than $ 2\left( \sum_{i=1,...,k-1} n_i \right) +r_k $, 
since the submatrix of the completion 
given by the first $ \sum_{i=1,...,k-1} n_i $ rows has rank less or equal than 
$ \sum_{i=1,...,k-1} n_i $ and the submatrix given by the last $n_k$ rows 
has rank less or equal than $ \sum_{i=1,...,k-1} n_i + r_k$ (in fact its
submatrix given by the first $ \sum_{i=1,...,k-1} n_i $ columns has rank 
less or equal than $ \sum_{i=1,...,k-1} n_i $ and the remaining part, that 
is $A_k$, has rank $r_k$).

\smallskip

Now we prove, by induction on $k$,
 that we can complete $A$ to a symmetric matrix whose rank is 
the number in (\ref{formula}). 
 By Remark \ref{oss}, we can suppose that 
$A_i$ is b-diagonal for $i=1,..., k$.

\underline{Case $k=2$.} 

Let $t= max\{n_1- r_1 , n_2 -r_2\}$.
Observe that $ t \leq n_1$ if and only if $ n_2 -r_2 \leq n_1$.

$\bullet $ If $t \leq n_1$, we consider the following symmetric 
completion of $A$:
$$ \left( \begin{array}{cc} A_1 & T^t_{n_1, n_2} 
\\  T^t_{n_2, n_1}  & A_2 \end{array}\right).$$  
By swapping the last $t$ rows of the upper blocks of the matrix 
with the last $t$ rows of the lower blocks of the matrix, we can see that 
the rank is $n_1 + n_2$.

$\bullet $
If $t > n_1$, then $t= n_2 -r_2$ and $n_2-r_2 > n_1$. In this case 
we consider the following symmetric completion of $A$:
$$ \left( \begin{array}{cc} A_1 & T^{n_1}_{n_1, n_2} 
\\  T^{n_1}_{n_2, n_1}  & A_2 \end{array}\right).$$  
By swapping the first $n_1$ rows of 
with the last $n_1$ rows, we can see that 
the rank is $n_1 + r_2 + n_1$.

\smallskip

So, in the case $k=2$ we have constructed a symmetric completion of rank
equal to $$\left\{ 
\begin{array}{ll} 
n_1 + n_2   & \mbox{\it if } \; n_2-r_2 \leq  n_1, \\ 
2 n_1 + r_2 & \mbox{\it if } \; n_2 -r_2 >  n_1, 
\end{array}
\right. $$
that is  $ min \{n_1 + n_2 , 2 n_1 + r_2\}$.

\underline{Induction step.} Let $k \geq 3$.

$\bullet $ If $ n_k-r_k \geq \sum_{i=1,..., k-1} n_i$, we can consider 
 the following symmetric completion: 
$$ \left( \begin{array}{ccccc|ccccccccc} 
A_1 & 0   & \cdot & \cdot  & 0        &   & & & &  & &         \\
0             & \ddots & \cdot  & \cdot  & \cdot   &   & & & &  & &    & &     \\
\cdot        & \cdot  & \ddots  & \cdot   & \cdot   & & & & &T^s_{s, n_k} &  & & &   \\
\cdot        & \cdot & \cdot  & \ddots   & 0       & &         & & & & &  & &    \\
0             & \cdot & \cdot  & 0    & A_{k-1} & &    & & & & &  & &         \\
\hline
 & & & & & & & & & & & & &    \\
 & & & & & & & & & & & & &    \\
 & & & & & & & & & & & & &    \\
 & & & & & & & & & & &  & &   \\
 & & T^{s}_{n_k, s} & &  & & & & & A_k & & &  \\
 & & & & & & & & &  & & & &   \\
 & & & & & & & & & & & & &    \\
 & & & & & & & & & & &  & &   \\
 & & & & & & & & &  & &   & &  
\end{array}\right),$$  
where $s= \sum_{i=1,..., k-1} n_i$. By swapping the first $s$ rows and the 
last $s$ rows, we can see
that its rank is $2s+ r_k=  2 (\sum_{i=1,..., k-1} n_i )
+r_k$.

$\bullet $ Suppose $ n_k-r_k < \sum_{i=1,..., k-1} n_i$. 
Let $P$ the submatrix of $A$ given by the last 
$\sum_{i=2,..., k} n_i$ rows and the last 
$\sum_{i=2,..., k} n_i$ columns.
By induction assumption, we can complete  $P$ 
to a symmetric matrix with rank 
$$ min \left\{ \sum_{i=2,...,k} n_i \; ,\; 2 \left( \sum_{i=2,...,k-1} n_i \right)
 +r_k\right\}.$$

We consider two cases: the case where we complete $P$ to a symmetric 
 matrix of rank $\sum_{i=2,...,k} n_i $ and the case where
 we complete $P$ to a symmetric matrix of rank $2 \left( \sum_{i=2,...,k-1} n_i
\right)+ r_k $. We  state that, in both cases,
 we can complete $A$ to a symmetric  matrix of rank
$\sum_{i=1,...,k} n_i $.

\smallskip

- Case  we complete $P$ to a symmetric  matrix of rank $\sum_{i=2,...,k} n_i $.

By a symmetric sequence of
 elementary operation $h$, 
we can change the completion of $P$ into a diagonal matrix $D$ where all the 
elements of the diagonal are nonzero; then we can complete 
$$ diag (A_1, D)$$ 
to the  symmetric matrix 
$$ \left( \begin{array}{cc} 
A_1   &  T^{n_1-r_1}_{n_1, \sum_{i=2,...,k} n_i }       \\
T^{ n_1-r_1}_{\sum_{i=2,...,k} n_i , n_1}     &  D
\end{array}
\right),$$  
which has rank $ \sum_{i=1,...,k} n_i$. 
Then, by applying the symmetric sequence
of elementary operations  $h^{-1}$,  
we get a symmetric completion of $A$ of rank $ \sum_{i=1,...,k} n_i$ .

- Case we complete $P$ to a symmetric 
 matrix of rank $2 \left( \sum_{i=2,...,k-1} n_i \right)+ r_k $.

By a symmetric sequence of
 elementary operation $h$, we can change the completion of $P$
into a diagonal matrix $L$ such that $ L_{i,i} \neq 0 $ if and only if 
$ i \leq 2 \left( \sum_{i=2,...,k-1} n_i \right) +r_k $;
define $$t= max \left\{n_1 -r_1 , 
\sum_{i=2,...,k} n_i - rk(L) \right\}= max \left\{n_1 -r_1 , 
\sum_{i=2,...,k} n_i - 2 \left( \sum_{i=2,...,k-1} n_i\right) -r_k \right\};$$ 
observe that 
$t \leq n_1$; then  we can complete $diag(A_1, L) $ 
to the  symmetric matrix
$$ \left( \begin{array}{cc} 
A_1   &  T^{t}_{n_1, \sum_{i=2,...,k} n_i }                \\
T^{t}_{\sum_{i=2,...,k} n_i, n_1 }     &  L
\end{array}
\right),$$  
which has rank $ \sum_{i=1,...,k} n_i$. Then, by applying the symmmetric 
sequence of elementary operations $h^{-1}$, 
we get a symmetric  completion of $A$ of rank $ \sum_{i=1,...,k} n_i$.
\end{proof}

\section{Completions of  antisymmetric block diagonal matrices}

\begin{rem}
(a) If $K$ is a  field of characteristic $2$, the set of the antisymmetric 
matrices $n \times n$  over $K$ is equal to the set of the matrices 
$n \times n$ over $K$.

(b)  If $K$ is a field of characteristic different from $2$,
we can change, by a symmetric sequence 
of elementary operations, an antisymmetric  matrix 
into a diagonal block matrix  whose diagonal blocks are all equal to 
$$ \left(\begin{array}{cc}
0 & 1 \\
-1 & 0
\end{array}\right).$$ 
In fact, if in the $i$-th row (and thus in the $i$-th column) there is 
a nonzero element,  by applying a symmetric sequence 
of elementary operations, we can suppose that the entry $(i,i+1) $ 
(and thus the entry $(i+1,i)$) is nonzero and all the other entries of 
the $i$-th row and the $i$-th column are zero.
By another symmetric sequence  of elementary operations,
we can get easily the  form described above.
\end{rem}

\begin{nota}
For every $ n \in \N$, we denote by $  \langle n \rangle $ 
the ``even part'' of $n$, that is 
$$ \langle n \rangle = 
\left\{ \begin{array}{ll}
n  & \mbox{\it if } n \;\mbox{\it  is even,} \\
n-1  & \mbox{\it if } n \; \mbox{\it  is odd.} 
\end{array} \right.$$ 
%Let $ Z= \left(\begin{array}{cc}
%0 & 1 \\
%-1 & 0
%\end{array}\right).$

\end{nota}

\begin{rem} \label{oss2} Let $K$ be a field of characteristic different
from $2$. 
Let $n_1,...., n_k$ be nonzero natural numbers.  
Let $A_i \in M(n_i \times n_i, K) $ be antisymmetric matrices 
for $i=1,..., k$ and let $r_i = rank (A_i)$.
Then there exists an antisymmetric 
completion of rank $l$ of the partial matrix $diag (A_1,...., A_k)$
 if and only if there exists an antisymmetric 
 completion of rank $l$ of the partial matrix 
$diag (R^{r_1}_{n_1},...., R^{r_k}_{n_k})$.
\end{rem}

(It can be proved as  Remark \ref{oss}.)

\begin{thm} \label{matantisim} 
Let $K$ be a field of characteristic different from $2$. 
Let $n_1,...., n_k$ be nonzero  natural numbers with
$n_1 \leq n_2 \leq .... \leq n_k$.  
Let $A_i \in M(n_i \times n_i, K) $ be antisymmetric 
matrices  for $i=1,..., k$
and let $r_i = rank (A_i)$ (obviously the $r_i$ are even numbers).
Let $A $ be the partial matrix $diag (A_1,...., A_k)$. 
Then we have:

(i) the minimum of $\{rk(\tilde{A}) | \; \tilde{A} \; \mbox{\it antisymmetric
completion of } A \}$ is $$ max\{r_i| \; i=1,...,k\}$$ 

(ii) the maximum of $\{rk(\tilde{A}) | \; \tilde{A} \; \mbox{\it antisymmetric
completion of } A \}$ is $$ min \left\{ \langle \sum_{i=1,...,k} n_i \rangle
,\;  2 \left(\sum_{i=1,...,k-1} n_i \right) +r_k
\right\}.$$ 
\end{thm}

\begin{proof}

(i) Let $A'_1,...., A'_k$ be the matrices $A_1,....,A_k$ ordered according to 
the rank, i.e.  let $A'_1,...., A'_k$ be such that  
$\{A'_1,...., A'_k\}= \{A_1,....,A_k\}$ and 
$s_1 \leq .... \leq s_k$,  where $ s_i =  rank (A'_i)$.
Obviously we can complete $diag (A'_1,...., A'_k)$ 
to an antisymmetric matrix of rank $l$ if and only if  we can complete 
$diag (A_1,...., A_k)$ to an antisymmetric matrix of rank $l$. 
Let $A' = diag (A'_1,..., A'_k)$.
By Remark \ref{oss2}, we can suppose that 
$A'_i= R^{s_i}_{m_i}$ for $i=1,..., k$, where $m_i$ is the number 
of the rows and of the columns of $ A'_i$. 

Observe that, since $s_i \leq s_{i+1}$, then $ s_i \leq m_{i+1}$. Thus  
we can complete the matrix $A'$ to the following antisymmetric matrix:
$$ \left( \begin{array}{ccccccc}
R^{s_1}_{m_1} & R^{s_1}  &  R^{s_1}  & \cdot & \cdot & \cdot & R^{s_1}  \\ 
R^{s_1}  & R^{s_2}_{m_2} &  R^{s_2}  & \cdot & \cdot & \cdot & R^{s_2}  \\
 R^{s_1} & R^{s_2}  &  R^{s_3}_{m_3} & \cdot & \cdot &  \cdot & R^{s_3} \\
\cdot & \cdot &  \cdot & \cdot & \cdot & \cdot & \cdot \\
\cdot & \cdot &  \cdot & \cdot & \cdot & \cdot & \cdot \\
\cdot & \cdot &  \cdot & \cdot & \cdot & \cdot & \cdot \\
R^{s_1} & R^{s_2} &  R^{s_3} & \cdot & \cdot & \cdot & R^{s_k}_{m_k}   
\end{array}\right),$$
where we omitted the sizes of the off-diagonal matrices  for simplicity 
(the sizes are obliged).
The rank of this completion is clearly $ max\{s_i| \; i=1,...,k\}$, which is 
equal to $ max\{r_i| \; i=1,...,k\}$.

Finally, observe that, obviously, the rank of any completion of 
$A$ is greater or equal than $ max\{r_i| \; i=1,...,k\}$.

(ii) The same argument as in Theorem \ref{matsim} proves that
any antisymmetric completion of $A$ has rank less or equal than 
$$ min \left\{ \langle \sum_{i=1,...,k} n_i \rangle , \;
 \; 2 \left( \sum_{i=1,...,k-1} n_i \right) +r_k \right\}.$$

\smallskip

Now we prove, by induction on $k$,  that we can complete $A$ to an 
antisymmetric matrix
whose rank is the number above.  By Remark \ref{oss2}, we can suppose that 
$A_i= R^{r_i}_{n_i}$ for $i=1,..., k$.

\underline{Case $k=2$.} 

Let $t= max\{n_1- r_1 , n_2 -r_2\}$.
Observe that $ t \leq n_1$ if and only if $ n_2 -r_2 \leq n_1$.

$\bullet $ If $t \leq n_1$, we consider the following 
antisymmetric completion of $A$:
$$ \left( \begin{array}{cc} R^{r_1}_{n_1} & -T^t_{n_1, n_2} 
\\  T^t_{n_2, n_1}  & R^{r_2}_{n_2} \end{array}\right).$$  
By swapping the last $t$ rows of the upper blocks of the matrix 
with the last $t$ rows of the lower blocks of the matrix, we can see that 
the rank is $\langle n_1 + n_2 \rangle $
(consider the two cases: 

1) both $n_1$ and $n_2$ are  odd or both of them are even  

2) one of $n_1$ and $n_2$ is odd and the other is even).

$\bullet $
If $t > n_1$, then $t= n_2 -r_2$ and $n_2-r_2 > n_1$. In this case 
we consider the following antisymmetric completion of $A$:
$$ \left( \begin{array}{cc} R^{r_1}_{n_1} & -T^{n_1}_{n_1, n_2} 
\\  T^{n_1}_{n_2, n_1}  & R^{r_2}_{n_2} \end{array}\right).$$  
By swapping the first $n_1$ rows of 
with the last $n_1$ rows, we can see that 
the rank is $n_1 + r_2 + n_1$.

\smallskip

So, in the case $k=2$ we have constructed a completion of rank
equal to $$\left\{ 
\begin{array}{ll} 
\langle n_1 + n_2  \rangle  & \mbox{\it if } \; n_2-r_2 \leq  n_1, \\ 
2 n_1 + r_2 & \mbox{\it if } \; n_2 -r_2 >  n_1, 
\end{array}
\right. $$
that is  $ min \{\langle n_1 + n_2 \rangle , 2 n_1 + r_2\}$.

\underline{Induction step.} Let $k \geq 3$.

$\bullet $ If $ n_k-r_k \geq \sum_{i=1,..., k-1} n_i$, we can consider 
 the following antisymmetric completion: 
$$ \left( \begin{array}{ccccc|ccccccccc} 
R^{r_1}_{n_1} & 0   & \cdot & \cdot  & 0        &   & & & &  & &         \\
0             & \ddots & \cdot  & \cdot  & \cdot   &   & & & &  & &    & &     \\
\cdot        & \cdot  & \ddots  & \cdot   & \cdot   & & & & &-T^s_{s, n_k} &  & & &   \\
\cdot        & \cdot & \cdot  & \ddots   & 0       & &         & & & & &  & &    \\
0             & \cdot & \cdot  & 0    & R^{r_{k-1}}_{n_{k-1}} & &    & & & & &  & &         \\
\hline
 & & & & & & & & & & & & &    \\
 & & & & & & & & & & & & &    \\
 & & & & & & & & & & & & &    \\
 & & & & & & & & & & &  & &   \\
 & & T^{s}_{n_k, s} & &  & & & & & R^{r_k}_{n_k} & & &  \\
 & & & & & & & & &  & & & &   \\
 & & & & & & & & & & & & &    \\
 & & & & & & & & & & &  & &   \\
 & & & & & & & & &  & &   & &  
\end{array}\right),$$  
where $s= \sum_{i=1,..., k-1} n_i$. By swapping the first $s$ rows with the
last $s$ rows, we can see that 
its rank is $2s+r_k= 2 \left( \sum_{i=1,..., k-1} n_i \right) +r_k$.

$\bullet $ Suppose $ n_k-r_k < \sum_{i=1,..., k-1} n_i$. 
By induction assumption, we can complete the submatrix $P$ of $A$
 given by the last 
$\sum_{i=2,..., k} n_i$ rows and the last 
$\sum_{i=2,..., k} n_i$  columns to an antisymmetric matrix with rank 
$$ min \left\{ \langle \sum_{i=2,...,k} n_i \rangle \; , \; 
 2 \left( \sum_{i=2,...,k-1} n_i \right) +r_k
\right\}.$$

We consider two cases: the case where we complete $P$ to a matrix of 
rank $ \langle \sum_{i=2,...,k} n_i \rangle $ and the case
we complete $P$ to a matrix of rank $2 
\left( \sum_{i=2,...,k-1} n_i \right) + r_k $.
We will show that, in both cases,
 we can complete $A$ to an antisymmetric matrix of rank
$ \langle \sum_{i=1,...,k} n_i \rangle $. 

\smallskip

- Case we complete $P$ to a matrix of rank $ \langle
\sum_{i=2,...,k} n_i \rangle $.

By a symmetric sequence of elementary operations $h$, 
 we can change the completion of $P$ into the  matrix
$ R^{\langle \sum_{i=2,...,k} n_i \rangle }_{\sum_{i=2,...,k} n_i}$. Then
 we can complete 
$$ diag \left(R^{r_1}_{n_1},   
R^{\langle \sum_{i=2,...,k} n_i \rangle }_{\sum_{i=2,...,k} n_i} \right)$$ 
to the antisymmetric matrix 
$$ \left( \begin{array}{cc} 
R^{r_1}_{n_1}   &  -T^{n_1-r_1}               \\
T^{ n_1-r_1}    &  R^{\langle \sum_{i=2,...,k} n_i \rangle }_{\sum_{i=2,...,k} n_i}
\end{array}
\right),$$  
which has rank $ \langle \sum_{i=1,...,k} n_i \rangle$, in fact:

if $ \sum_{i=2,...,k} n_i $ is even, the rank is
$$ \left\{ \begin{array}{ll}
\sum_{i=1,..., k} n_i & \mbox{\it  if } n_1-r_1 \; \mbox{\it  is even, } \\ 
\sum_{i=1,..., k} n_i -1 & \mbox{\it  if } n_1-r_1 \; \mbox{\it  is odd, } 
 \end{array} \right.$$

 if $ \sum_{i=2,...,k} n_i $ is odd, the rank is 
$$ \left\{ \begin{array}{ll}
\sum_{i=1,..., k} n_i -1 & \mbox{\it  if } n_1-r_1 \; \mbox{\it  is even, } \\ 
\sum_{i=1,..., k} n_i  & \mbox{\it  if } n_1-r_1 \; \mbox{\it  is odd. } 
 \end{array} \right.$$

Then, by applying the symmetric sequence of elementary operations $h^{-1}$, 
we get an antisymmetric 
 completion of $A$ of rank $\langle \sum_{i=1,...,k} n_i \rangle$ .

- Case we complete $P$ to a matrix of rank $2 \left( \sum_{i=2,...,k-1} n_i
\right) + r_k $.

By a symmetric sequence of elementary operations $h$, 
we can change the completion of $P$ into 
$ R^{2(\sum_{i=2,...,k-1} n_i) +r_k }_{\sum_{i=2,...,k} n_i}$;
define $$t= max \left\{n_1 -r_1 , 
\sum_{i=2,...,k} n_i - 2 \left( \sum_{i=2,...,k-1} n_i\right) -r_k \right\};$$ observe that 
$t \leq n_1$; then  we can complete $diag \left(R^{r_1}_{n_1}, 
R^{2\sum_{i=2,...,k-1} n_i +r_k }_{\sum_{i=2,...,k} n_i}\right) $ 
to the antisymmetric matrix 
$$ \left( \begin{array}{cc} 
R^{r_1}_{n_1}   &  -T^{t}               \\
T^{t}    &  R^{2\sum_{i=2,...,k-1} n_i +r_k}_{\sum_{i=2,...,k} n_i}
\end{array}
\right),$$  
which has rank $ \langle \sum_{i=1,...,k} n_i \rangle$; then, by applying the 
symmetric sequence of 
elementary operations  $h^{-1}$,   we get a completion of $A$ of rank
$ \langle \sum_{i=1,...,k} n_i \rangle$.

\end{proof}

{\footnotesize\em Dipartimento di Matematica e Informatica ``U. Dini'', 
viale Morgagni 67/A,
50134  Firenze, Italia }

{\footnotesize\em
E-mail address: rubei@math.unifi.it}

\end{document}